\documentclass[oneside,10pt]{amsart}
\usepackage{amsmath, amsfonts,amsthm,times,graphics}

 \makeatletter
\renewcommand*\subjclass[2][2010]{%
  \def\@subjclass{#2}%
  \@ifundefined{subjclassname@#1}{%
    \ClassWarning{\@classname}{Unknown edition (#1) of Mathematics
      Subject Classification; using '2010'.}%
  }{%
    \@xp\let\@xp\subjclassname\csname subjclassname@#1\endcsname
  }%
}

 \makeatother
\newtheorem{theorem}{Theorem}[section]
\newtheorem{lemma}[theorem]{Lemma}
\newtheorem{corollary}[theorem]{Corollary}
\newtheorem{proposition}[theorem]{Proposition}
\theoremstyle{definition}

\newtheorem{remark}[theorem]{Remark}
\numberwithin{equation}{section}

 \makeatletter
\renewcommand*\subjclass[2][2010]{%
  \def\@subjclass{#2}%
  \@ifundefined{subjclassname@#1}{%
    \ClassWarning{\@classname}{Unknown edition (#1) of Mathematics
      Subject Classification; using '1991'.}%
  }{%
    \@xp\let\@xp\subjclassname\csname subjclassname@#1\endcsname
  }%
}

 \makeatother

\begin{document}
\title{Clarkson's type inequalities for positive
 $l_p$ sequences with $p\ge 2$}

\author{Romeo Me\v strovi\' c}
\address{Maritime Faculty, University of Montenegro, Dobrota 36,
85330 Kotor, Montenegro} \email{romeo@ac.me}

{\renewcommand{\thefootnote}{}%

\begin{abstract} 
For a fixed $1\le p<+\infty$ denote by $\Vert\cdot\Vert_p$ 
the usual norm in the space $l_p$ (or $L_p$).
In this paper we prove that for all real numbers $p$ and $q$
such that $2\le p\le q$ holds
  $$
2(\Vert x\Vert_p^q+\Vert y\Vert_p^q)\le 
\Vert x+y\Vert_p^q  +\Vert x-y\Vert_p^q
  $$
for all nonnegative sequences $x=\{x_n\},y=\{y_n\}$ in $l_p$ (or 
nonnegative functions $x,y$ in $L_p$).
Note that the above inequality
with $p=q\ge 2$ reduces to the well known Clarkson's  inequality.
If in addition, holds $x_i\ge y_i$ for each $i=1,2,\ldots$ (or
$x\ge y$ a.e. in $L_p$), 
then we establish an improvement of the above inequality.

\end{abstract}
\maketitle
              \section{Introduction}
Let $(X,\Sigma,\mu)$ be a measure space with a positive Borel
measure $\mu$. For any $0<p<+\infty$ let $L^p=L^p(\mu)$ denotes the
usual Lebesgue space consisting of all  $\mu$-measurable complex-valued 
functions  $f: X \to\Bbb{C}$  such that
    $\int_X|f|^p\,d\mu<+\infty.$
Recall that for any $p\ge 2$ the usual norm 
$\Vert\cdot\Vert_p$ of $f\in L^p$ is
defined as $\Vert f\Vert_p=\left(\int_X|f|^p\,d\mu\right)^{1/p}$.
Further, the space $l_p$ is the space of all 
complex sequences $a=\{a_n\}$ such that $\sum_{n=1}^{+\infty}
|a_n|^p<+\infty$;
if $2\le p<+\infty$, then the usual 
$l_p$-norm is defined for $a\in l_p$ as $\Vert a\Vert_p:=
\left(\sum_{n=1}^{+\infty}|a_n|^p\right)^{1/p}$. 

In \cite{c} J. Clarkson introduced the concept of uniform convexity in 
Banach spaces and obtained that the spaces $L_p$ 
 (or $l_p$) with $1< p<+\infty$ are uniformly 
convex  
throughout the following inequalities.
    \begin{theorem}$($\cite[Theorem 2]{c}$)$
Let $2\le p<+\infty$, and let $q=p/(p-1)$ be the conjugate of 
$p$. Then for all $x$ and $y$ in $L_p$ or $l_p$, 
  \begin{equation}\label{ineq1.1}
2(\Vert x\Vert_p^p+\Vert y\Vert_p^p)^{q-1}\le 
\Vert x+y\Vert_p^q  +\Vert x-y\Vert_p^q
  \end{equation}
   \begin{equation}\label{ineq1.2}
 \Vert x+y\Vert_p^p  +\Vert x-y\Vert_p^p
\le 2(\Vert x\Vert_p^q+\Vert y\Vert_p^q)^{p-1}
    \end{equation}
  \begin{equation}\label{ineq1.3}
2(\Vert x\Vert_p^p+\Vert y\Vert_p^p)\le 
\Vert x+y\Vert_p^p  +\Vert x-y\Vert_p^p\le 
2^{p-1}(\Vert x\Vert_p^p+\Vert y\Vert_p^p).
    \end{equation}  
For $1<p\le 2$ these inequalities hold in the reverse sense.
 \end{theorem}
Clarkson pointed out that for all values
of $p$  the right hand side of  \eqref{ineq1.3} is
equivalent to the left hand side, while \eqref{ineq1.1}
is equivalent to \eqref{ineq1.2}; to see this,
set $x+y=u$, $x-y=y$ and reduce.
Notice that the proof of \eqref{ineq1.2}
 (\cite[Proof of Theorem 2]{c}) is rather long and nontrivial. 
In this proof Clarkson deduced \eqref{ineq1.3} from \eqref{ineq1.1}.
A direct simple  proof of the inequality  \eqref{ineq1.3},  
based on the classical  H\"{o}lder inequality was given 
in \cite{r}.

For each $1\le p<+\infty$ 
denote by $l_p^+$ the set of all nonnegative 
real sequences in $l_p$,
and by  $L_p^+$ the set of all nonnegative real-valued functions in $L_p$.
In this note we prove an extension of  the inequality \eqref{ineq1.3}
for the spaces $l_p^+$ (or $L_p^+$) with $p\ge 2$ as follows.
   \begin{theorem}
Let $2\le p\le q<+\infty$. Then for all $x$ 
and $y$ in $l_p^+$ (or $L_p^+$) we have
   \begin{equation}\label{ineq1.7}
2(\Vert x\Vert_p^q+\Vert y\Vert_p^q)
\le\Vert x+y\Vert_p^q  +\Vert x-y\Vert_p^q. 
    \end{equation}  
    \end{theorem}

\begin{remark}
Note that for $p=q\ge 2$  the inequality \eqref{ineq1.7} 
reduces to the Clarkson's inequality on the left hand side of
 \eqref{ineq1.3}.

On the other hand, if $2\le p\le q<+\infty$, 
then $1/p+1/q=1$ only for $p=q=2$, and
thus  the inequality \eqref{ineq1.7} 
cannot be derived from any Clarkson's inequalities
in Theorem 1.1. 
\end{remark} 

The following result is basic for the proof of Theorem 1.2.

\begin{proposition}
Let $2\le p\le q<+\infty$ and let 
$u=(u_1,u_2,\ldots,u_n,\ldots)$ and $v=(v_1,v_2,\ldots,v_n,\ldots)$ 
be nonnegative sequences  in $l_p^+$ such that $u_i\ge v_i$ for each 
$i=1,2,\ldots n,\ldots$  (or $u$ and $v$ be functions 
in $L_p^+$ such that $u\ge v$ a.e. on $X$). Then
   \begin{equation}\label{ineq1.4}
\Vert u+v\Vert_p^q  +\Vert u-v\Vert_p^q
\ge 2(\Vert u\Vert_p^q+2^{q-2}\Vert v\Vert_p^q).
  \end{equation}
In particular, for $q=p\ge 2$, we have
    \begin{equation}\label{ineq1.5}
\Vert u+v\Vert_p^p  +\Vert u-v\Vert_p^p
\ge 2(\Vert u\Vert_p^p+2^{p-2}\Vert v\Vert_p^p).
  \end{equation} 
   \end{proposition}
 As an immediate consequence, we obtain 
an improvement of the Clarkson's inequality
on the left hand side of \eqref{ineq1.3} in the real case.
   \begin{corollary}
For real numbers $x\ge y\ge 0$ and $q\ge 2$ holds
     \begin{equation}\label{ineq1.6}
2(x^q+2^{q-2}y^q)\le(x+y)^q  +(x-y)^q. 
  \end{equation} 
    \end{corollary}
The proofs of Proposition 1.4 and Theorem 1.2
are given in the next section and they are  very simple and 
elementary.
             \section{Proofs of Proposition 1.4  and Theorem 1.2} 
  The following interesting lemma is basic in the proofs of
Proposition 1.4. 

\begin{lemma} Let $2\le p\le q<+\infty$ and let 
$u=(u_1,u_2,\ldots,u_n)$ and 
$v=(v_1,v_2,\ldots,v_n)$ be  $n$-tuples 
of nonnegative real numbers such that $u_i\ge v_i$ for each $i=1,2,\ldots,n$.
Define the function $\varphi$ from $[0,1]$ into $\Bbb{R}$
by 
  \begin{equation}\label{ineq2.1}
\varphi(t)=\Vert u+tv\Vert_p^q  +\Vert u-tv\Vert_p^q-2^{q-1}
(\Vert u\Vert_p^q+\Vert tv\Vert_p^q),\quad t\in [0,1].
    \end{equation}
Then $\varphi$ is a nondecreasing  function on $[0,1]$.
  \end{lemma}
\begin{proof}
Obviously, \eqref{ineq2.1} may be written as
    \begin{equation}\label{ineq2.2}\begin{split}
\varphi(t)=&\big(\sum_{i=1}^n|u_i+v_it|^p\big)^{q/p}+
\big(\sum_{i=1}^n|u_i-v_it|^p\big)^{q/p}\\
&-2^{q-1}\Big(\big(\sum_{i=1}^n|u_i|^p\big)^{q/p}+
\big(\sum_{i=1}^n|v_i|^p\big)^{q/p}t^q\Big).
  \end{split}\end{equation}
Note that the function $\varphi$ is continuous on $[0,1]$
and differentiable on the set $D:=(0,1)\setminus\{\pm u_i/v_i:\, v_i\not=0,
i=1,2,\ldots, n\}.$ Since for 
any fixed $1\le i\le n$ hold $u_i+v_it\ge 0$ and $u_i-v_it \ge u_i-v_i\ge 0,$
we have, $\frac{d}{dt}|u_i+v_it|=v_i$, and 
$\frac{d}{dt}|u_i-v_it|=-v_i$ for each $t\in D$. Hence,
 the derivative of $\varphi$ for $t\in D$ is
   \begin{equation}\label{ineq2.3}\begin{split}
\varphi'(t)=&q\left(\Big(\sum_{i=1}^n(u_i+v_it)^p\Big)^{(q/p)-1}\cdot
\Big(\sum_{i=1}^nv_i(u_i+v_it)^{p-1}\Big)\right.\\
&-\Big(\sum_{i=1}^n(u_i-v_it)^p\Big)^{(q/p)-1}\cdot
\Big(\sum_{i=1}^nv_i(u_i-v_it)^{p-1}\Big)\\
&-2^{q-1}\left.\Big(\sum_{i=1}^nv_i^p\Big)^{q/p}t^{q-1}\right).
       \end{split}\end{equation}
We will show that $\varphi'(t)\ge 0$ for each $t\in D$.
To show this, we use the well known inequality 
$(a+b)^r\ge a^r+b^r$ for real numbers $a\ge 0$, $b\ge 0$ and $r\ge 1$.
Applying this inequality in the form 
$(u_i+v_it)^r\ge (u_i-v_it)^r+(2v_it)^r$ for $1\le i\le n$,
 with $r\in\{p,p-1\}$, and using that $q/p \ge 1$,
 we estimate the first term on the right of \eqref{ineq2.3} as 
     \begin{equation}\label{ineq2.4}\begin{split}
&\Big(\sum_{i=1}^n(u_i+v_it)^p\Big)^{(q/p)-1}\cdot
\Big(\sum_{i=1}^nv_i(u_i+v_it)^{p-1}\Big)\\
\ge& \Big(\sum_{i=1}^n(u_i-v_it)^p+\sum_{i=1}^n(2v_it)^p\Big)^{(q/p)-1}\cdot
\Big(\sum_{i=1}^nv_i(u_i-v_it)^{p-1}+\sum_{i=1}^nv_i(2v_it)^{p-1}\Big)\\
= &\Big(\sum_{i=1}^n(u_i-v_it)^p+2^pt^p\sum_{i=1}^nv_i^p\Big)^{(q/p)-1}\cdot
\Big(\sum_{i=1}^nv_i(u_i-v_it)^{p-1}\Big)\\
&+ \Big(\sum_{i=1}^n(u_i-v_it)^p+2^pt^p\sum_{i=1}^nv_i^p\Big)^{(q/p)-1}\cdot
\Big(2^{p-1}t^{p-1}\sum_{i=1}^nv_i^{p}\Big)\\
\ge & \Big(\sum_{i=1}^n(u_i-v_it)^p\Big)^{(q/p)-1}\cdot
\Big(\sum_{i=1}^nv_i(u_i-v_it)^{p-1}\Big)\\
&+ \Big(2^pt^p\sum_{i=1}^nv_i^p\Big)^{(q/p)-1}\cdot
\Big(2^{p-1}t^{p-1}\sum_{i=1}^nv_i^{p}\Big)\\
= & \Big(\sum_{i=1}^n(u_i-v_it)^p\Big)^{(q/p)-1}\cdot
\Big(\sum_{i=1}^nv_i(u_i-v_it)^{p-1}\Big)
+ 2^{q-1}\Big(\sum_{i=1}^nv_i^p\Big)^{q/p}t^{q-1}.
    \end{split}\end{equation}
Now inserting the inequality \eqref{ineq2.4} into  \eqref{ineq2.3},
we immediately obtain that for each $t\in D$
 \begin{equation}\label{ineq2.5}
\varphi'(t)\ge q\left(2^{q-1}\Big(\sum_{i=1}^nv_i^p\Big)^{q/p}t^{q-1}
-2^{q-1}\left(\sum_{i=1}^nv_i^p\right)^{q/p}t^{q-1}\right)=0.
  \end{equation}
Therefore, $\varphi'(t)\ge 0$ for each $t\in D$,
and since $\varphi$ is a continuous function on $[0,1]$,
we infer that $\varphi$ is a nondecreasing function on $[0,1]$,
and the proof is completed.
\end{proof}

\begin{proof}[Proof of Proposition 1.4.]
We first consider the case $2\le p\le q<+\infty$
related to  the space $l_p$. By the continuity of $l_p$-norm,
it is enough to prove the inequality \eqref{ineq1.4} for two arbitrary
$n$-tuples  $u=(u_1,u_2,\ldots,u_n)$ and 
$v=(v_1,v_2,\ldots,v_n)$ of nonnegative real numbers with $u_i\ge v_i$ for
each $i=1,2,\ldots,n$. Since the function 
$\varphi$ defined by \eqref{ineq2.1} from Lemma 2.1  increases on $[0,1]$, 
we have
    \begin{equation}\label{ineq2.6}
\Vert u+v\Vert_p^q  +\Vert u-v\Vert_p^q-
2^{q-1}(\Vert u\Vert_p^q+\Vert v\Vert_p^q)=\varphi(1)\ge \varphi(0)=
2\Vert u\Vert_p^q-2^{q-1}\Vert u\Vert_p^q,
  \end{equation}
or equivalently,
     \begin{equation}\label{ineq2.7}
\Vert u+v\Vert_p^q  +\Vert u-v\Vert_p^q
\ge 2(\Vert u\Vert_p^q+2^{q-2}\Vert v\Vert_p^q),
  \end{equation}
as desired.

In order to prove that the inequality \eqref{ineq1.4} can be extended
for any two functions $u,v\in L_p^+$ such that $u\ge v$ a.e. in $L_p^+$,
 it is enough to apply a 
standard argument that the set  of all measurable simple functions
$s:X\to [0,+\infty)$ forms a dense subset in $L_p$,
and so the result follows by the continuity of the norm
(see \cite[Proof of Theorem 2]{c}; also cf. \cite[Proof of Theorem 2.3]{si}).
\end{proof}

For the proof of Theorem 1.2 we will need still  two following
lemmas.

\begin{lemma} Let $A,a,B,b$ and $r\ge 1$ be nonnegative real numbers 
such that $A\ge B$ and $a> b$. Then
   \begin{equation}\label{ineq2.8}
(A+a)^r+(B+b)^r > (A+b)^r+(B+a)^r.
  \end{equation}
\end{lemma}
\begin{proof}
Put $c=a+b$, and define the function $f$ from $[c/2,c]$ into $\Bbb{R}$
by 
  \begin{equation}\label{ineq2.9}
f(t)=(A+t)^r+(B+c-t)^r,\quad t\in \left[\frac{c}{2},c\right].
  \end{equation} 
Then
 \begin{equation}\label{ineq2.10}
f'(t)=r((A+t)^{r-1}-(B+c-t)^{r-1}),\quad t\in \left[\frac{c}{2},c\right].
 \end{equation} 
Since $A\ge B$, we see that $A+t> B+c-t\ge 0$ for $t\in (c/2,c]$;
so, for such a  $t$, we have  $(A+t)^{r-1}-(B+c-t)^{r-1}> 0$.
This shows that $f$ is an increasing function on  $(c/2,c]=((a+b)/2,a+b]$.
This together with the fact that $(a+b)/2< b< a\le a+b$ yields
$f(a)> f(b)$, which is actually the inequality \eqref{ineq2.8}. 
   \end{proof}

    \begin{lemma} Let $r\ge 1$ and let $x=(x_1,x_2,\ldots,x_n)$ and 
$y=(y_1,y_2,\ldots,y_n)$ be $n$-tuples 
of nonnegative real numbers.
Define sequences $x'=(u_1,u_2,\ldots,u_n)$ and $y'=(v_1,v_2,\ldots,v_n)$
as 
   \begin{equation}\label{ineq2.11}\begin{split}
 u_i=x_i \,\,\, and \,\,\, v_i & =y_i \,\,\, if\,\,\, x_i\ge y_i\\
u_i=y_i \,\,\, and \,\,\, v_i & =x_i \,\,\, if\,\,\, x_i< y_i
    \end{split}\end{equation}
for each $i=1,2,\ldots n$. Then 
   \begin{equation}\label{ineq2.12}
\left(\sum_{i=1}^n u_i\right)^r+\left(\sum_{i=1}^n y_i\right)^r\ge
\left(\sum_{i=1}^n x_i\right)^r+\left(\sum_{i=1}^n y_i\right)^r.
  \end{equation}
\end{lemma}
    \begin{proof} Without loss of generality, we can suppose that 
      \begin{equation}\label{ineq2.13}
S:=\left(\sum_{i=1}^n x_i\right)^r\ge \left(\sum_{i=1}^n y_i\right)^r:=T.
    \end{equation}
Clearly, in order to prove the inequality \eqref{ineq2.12}, 
it is suffices to show that if for some index $k$ holds $x_k<y_k$, 
then after  the interchange of $x_k$ and  $y_k$ the sum on the right hand side
of \eqref{ineq2.12} increases. Assume that  $x_k<y_k$ for some
$k\in\{1,2,\ldots n\}$, and denote new sums by
 \begin{equation}\label{ineq2.14}
S':=\left(\sum_{i\not=k} x_i+y_k\right)^r=
\left(\sum_{i=1}^nx_i-x_k+y_k\right)^r
\end{equation}
  \begin{equation}\label{ineq2.15}
T':=\left(\sum_{i\not=k} y_i+x_k\right)^r=
\left(\sum_{i=1}^ny_i-y_k+x_k\right)^r.
\end{equation}
If we denote $A=\sum_{i=1}^nx_i-x_k$ and $B=\sum_{i=1}^ny_i-y_k$,
then from  \eqref{ineq2.13} and $x_k<y_k$ we see that $A>B\ge 0$.
Now applying Lemma 2.2 with $a=y_k$ and $b=x_k$, we find that
   \begin{equation}\begin{split}\label{ineq2.16}
S'+T'&=(A+y_k)^r+(B+x_k)^r> (A+x_k)^r+(B+y_k)^r\\
&=(\sum_{i=1}^nx_i-x_k+x_k)^r+(\sum_{i=1}^ny_i-y_k+y_k)^r=
S+T.
  \end{split}\end{equation}
Hence, repeating the above procedure succesufely  for all  indices 
$k\in\{1,2,\ldots ,n\}$ such that $x_k<y_k$, we obtain 
our inequality \eqref{ineq2.12}.
\end{proof}
 
\begin{proof}[Proof of Theorem 1.2.]

We first consider the case $2\le p\le q<+\infty$
related to  the sequence space $l_p$. By the continuity of $l_p$-norm,
it is enough to prove the inequality \eqref{ineq1.7} for any two 
nonnegative $n$-tuples  $x=(x_1,x_2,\ldots,x_n)$ and 
$y=(y_1,y_2,\ldots,y_n)$. 
Because of symmetry, without loss of generality we can suppose 
that $\Vert x\Vert_p\ge \Vert y\Vert_p$.
 Accordingly, we can assume that $x_1\ge y_1$.
This fact, together with Lemma 2.3  with $x_i^p$ and $y_i^p$ 
instead of $x_i$ and $y_i$, respectively,
and with $u_i$, $v_i$, defined by \eqref{ineq2.11} 
($i=1,2,\ldots,n$),  yields that
    \begin{equation}\begin{split}\label{ineq2.17}
\big(\sum_{i=1}^n|u_i|^p\big)^{q/p}+\big(\sum_{i=1}^n|v_i|^p\big)^{q/p}
&\ge \big(\sum_{i=1}^n|x_i|^p\big)^{q/p}+\big(\sum_{i=1}^n|y_i|^p\big)^{q/p}\\
&= \Vert x\Vert_p^q+\Vert y\Vert_p^q.
  \end{split}\end{equation}
Obviously, the expression $\Vert x+y\Vert_p^q  +\Vert x-y\Vert_p^q$
is invariant under the interchange $x_i$ with $y_i$
for any $i=1,2,\ldots,n$, and therefore,
  \begin{equation}\label{ineq2.18}
\Vert x+y\Vert_p^q  +\Vert x-y\Vert_p^q=
\big(\sum_{i=1}^n|u_i+v_i|^p\big)^{q/p}+
\big(\sum_{i=1}^n|u_i-v_i|^p\big)^{q/p}.
  \end{equation}
Further, since $q\ge 2$, the inequality \eqref{ineq1.4} in Theorem 1.2
implies that   
  \begin{equation}\label{ineq2.19}
\big(\sum_{i=1}^n|u_i+v_i|^p\big)^{q/p}+
\big(\sum_{i=1}^n|u_i-v_i|^p\big)^{q/p}\ge 
2\big(\sum_{i=1}^n|u_i|^p\big)^{q/p}+\big(\sum_{i=1}^n|v_i|^p\big)^{q/p}.
  \end{equation}
Clearly, the relations \eqref{ineq2.17}, \eqref{ineq2.18} and \eqref{ineq2.19} 
immediately imply the inequality  \eqref{ineq1.7}, that is, 
  \begin{equation}\label{ineq2.20}
\Vert x+y\Vert_p^q  +\Vert x-y\Vert_p^q\ge 
2(\Vert x\Vert_p^q+\Vert y\Vert_p^q).
    \end{equation}  
Finally, recall that the inequality \eqref{ineq1.7}
may be extended to the space $L_p^+$, by applying the same argument
as noticed at the end of proof of Theorem 1.2.
\end{proof}
  \begin{remark}
Note that the method in our proof of Theorem 1.2 
cannot be applied for the proof of Clarkson's inequalities 
\eqref{ineq1.1}  and \eqref{ineq1.2}. For example, 
Clarkson's inequality  \eqref{ineq1.1} in the real case may be written as
  \begin{equation}\label{ineq2.21}
|x+y|^q+|x-y|^q\ge 2(|x|^p+|y|^p)^{q-1},\quad x,y\in\Bbb{R}, p=q/(q-1)\ge 2.
   \end{equation}
Assuming that $|x|\le |y|$ and setting $c= |x|/|y|\le 1$, 
the inequality \eqref{ineq2.21} is equivalent to
  \begin{equation}\label{ineq2.23}
(1+c)^q+(1-c)^q\ge 2(1+c^{p)})^{q-1},
 \quad 0\le c\le 1.
   \end{equation}
Next for a fixed $c\in [0,1]$,  consider the function $\psi$
 from $[0,1]$ into $\Bbb{R}$ defined by 
 \begin{equation}\label{ineq2.24}
\psi(t)=(1+ct)^q+(1-ct)^q- 2(1+c^pt^{p)})^{q-1},\quad t\in [0,1]
   \end{equation}
A direct calculation gives 
 \begin{equation}\label{ineq2.25}
\psi'(t)=qc(1+ct)^{q-1}-qc(1-ct)^{q-1}-2p(q-1)(1+c^pt^{p})^{q-2}c^pt^{p-1},
 t\in [0,1],
   \end{equation}
which by using that $p(q-1)=q$ and  setting $s=ct$, can be written as
   \begin{equation}\label{ineq2.26}
\psi'(t)=\chi (s)=qc((1+s)^{q-1}-(1-s)^{q-1}-2(1+s^{p})^{q-2}s^{p-1}),
s\in [0,c].
   \end{equation}
However,  direct calculations show that for several values of $p$ and $q$
with $2\le p\le q$ and a fixed $0<c\le 1$, the function $\chi(s)$ takes 
positive and negative values on the segment $[0,c]$. Thus, generally,
the function $\psi(t)$ is not monotone on $[0,c]$.
       \end{remark}

\end{document}